\documentclass[11pt,reqno]{amsart}
\usepackage{mathrsfs} 
\usepackage{amsfonts}
\usepackage{amssymb,color}
\usepackage{amssymb}
\usepackage{amsthm}

\usepackage[colorlinks=true]{hyperref}
\hypersetup{urlcolor=blue, citecolor=blue}

\usepackage[margin=3cm, a4paper]{geometry}

\def\normo#1{\left\|#1\right\|}

\def\aabs#1{\left|#1\right|}
\def\brk#1{\left(#1\right)}

\def\norm#1{\|#1\|}
\def\jb#1{\langle#1\rangle}
\def\wt#1{\widetilde{#1}}
\def\wh#1{\widehat{#1}}

\newcommand{\N}{{\mathbb N}}

\newcommand{\R}{{\mathbb R}}

\newcommand{\Z}{{\mathbb Z}}
\newcommand{\ft}{{\mathcal{F}}}

\newcommand{\les}{{\lesssim}}
\newcommand{\ges}{{\gtrsim}}

\newcommand{\supp}{{\mbox{supp}}}
\newcommand{\sgn}{{\mbox{sgn}}}

\newcommand{\Lr}{{\mathscr{L}}}

\def\jb#1{\langle#1\rangle}
\def\norm#1{\|#1\|}
\def\normo#1{\left\|#1\right\|}

\def\wt#1{\widetilde{#1}}
\def\wh#1{\widehat{#1}}
\def\aabs#1{\left|#1\right|}

\newcommand{\F}{\mathcal{F}}

\newcommand{\mA}{\mathbf{A}}
\newcommand{\mB}{\mathbf{B}}
\newcommand{\mG}{\mathbf{G}}
\newcommand{\mI}{\mathbf{I}}
\newcommand{\hess}{{\mathbf{H}}}

\newcommand{\al}{\alpha}
\newcommand{\be}{\beta}
\newcommand{\ga}{\gamma}
\newcommand{\de}{\delta}
\newcommand{\e}{\varepsilon}

\newcommand{\p}{\partial}

\newcommand{\EQ}[1]{\begin{equation}\begin{split} #1 \end{split}\end{equation}}
\newcommand{\EQN}[1]{\begin{equation*}\begin{split} #1 \end{split}\end{equation*}}
\newcommand{\Del}[1]{}

\numberwithin{equation}{section}

\newtheorem{thm}{Theorem}[section]
\newtheorem{cor}[thm]{Corollary}
\newtheorem{lem}[thm]{Lemma}
\newtheorem*{lem*}{Lemma}

\theoremstyle{remark}

\newtheorem{ex}{Example}

\newtheorem{theorem}{Theorem}[section]

\theoremstyle{remark}
  \newtheorem{remark}[theorem]{Remark}

\theoremstyle{definition}

\begin{document}

\title[Oscillatory integrals]{Uniform estimates for oscillatory integrals with parameter-dependent phases}
\author[Z. Guo]{Zihua Guo}

\address{School of Mathematics,  Monash University,  Clayton VIC 3800, Australia}
\email{zihua.guo@monash.edu}

\begin{abstract}
We consider the oscillatory integrals with parameter-dependent phases.  We decompose the integrals into a leading term and a remainder term.   Instead of the pointwise estimate,  we use some $L^p$-estimate for the remainder term and get various uniform estimates when the phase functions satisfy certain conditions.  This enables us to reduce the requirement of the smoothness on the phase functions,  and hence improve the results in \cite[Theorem 7.7.5]{Hormander} and also obtain a refined version of the well-known Van der Corput Lemma.   Some applications on the uniform expansion of the Bessel functions and dispersive estimates are also given.
\end{abstract}

\maketitle


\section{Introduction}

Consider the oscillatory
integral
\[I(\lambda)=\int_{\R^d} e^{i\lambda \phi(x)}\psi(x)dx, \quad \lambda\in \R^+\]
where $\psi\in C_0^\infty(\R^d)$ and the real-valued phase function $\phi\in C^\alpha$ on $\supp(\psi)$.   We are
interested in the behaviour of $I(\lambda)$ as $\lambda\to \infty$.  Due to the oscillation of the integrand, we expect $|I(\lambda)|$ decays as $\lambda\to \infty$.  For example, if the phase is non-stationary and smooth,  namely $\nabla \phi(x)\neq 0, \ x\in \supp (\psi)$, then one can use integration by part to get the decay estimate $|I(\lambda)|\les \lambda^{-N}$ for all $N\in \N$, see Stein \cite[Proposition 1 in Chapter VIII]{Stein}.  

On the other hand, in the case when the phase has stationary point 
(i.e., critical point such that $\nabla\phi(x)=0$), the oscillatory integral has been extensively studied. One approach is the method of stationary phase, that is, if $\phi$ has a non-degenerate critical point at $x_0$ and $\psi$ is sufficiently localized at $x_0$, then we have that (see Stein \cite[Proposition 6 in Chapter VIII]{Stein})
\begin{align}\label{eq:Ilambda}
\left|I(\lambda)-\frac{e^{i\lambda \phi(x_0)}\psi(x_0)(i2\pi)^{d/2}\lambda^{-d/2}}{(\det \hess \phi(x_0))^{1/2}}\right|\leq C_{\phi,\psi}|\lambda|^{\frac{-d-1}{2}}.
\end{align}
Here $\hess \phi(x_0)=[\frac{\p^2\phi}{\p x_i \p x_j}(x_0)]$ is the Hessian matrix of $\phi$ at $x_0$ and $C_{\phi,\psi}$ is a constant depending on $\phi,\psi$. Moreover, an asymptotic expansion for $I(\lambda)$ was given in \cite[Proposition 6 in Chapter VIII]{Stein}, however, the dependence of $C_{\phi,\psi}$ on $\phi,\psi$ was not explicitly stated.  

The oscillatory integral has a wide range of applications to areas such as partial differential equations,  special functions,  etc.   For example, it plays a crucial role in dispersive equations.  See \cite{KPV, decay}.   In many applications, $\phi$ may be parameter-dependent. In some situations $\phi$ may even depend on $\lambda$.  For example,  for some problems we need to get the decay in $\lambda$ of the oscillatory integral
\[\int e^{i\varphi(x,\lambda)}\psi(x)dx.\]
When the function $\varphi$ is inhomogeneous in $\lambda$,  we do not have a neat form $\varphi(x,\lambda)=\lambda^\theta \phi(x)$ for some $\phi$.  So we write it as $\varphi(x,\lambda)=\lambda \cdot \varphi(x,\lambda)\lambda^{-1}$, and apply the method of stationary phase to the phase function $\phi_\lambda=\varphi(x,\lambda)\lambda^{-1}$.  A typical example is the Bessel function (see \eqref{eq:Besselint}) whose main term is
\[J_\nu^M(r):=\int_{-\pi}^\pi e^{i(r\sin x-\nu
x)}dx.\]
The phase function $r\sin x-\nu x$ is not homogeneous in $r$ and depends on $\nu$. The asymptotic behaviour of $J_\nu^M(r)$ as $r\to \infty$ is crucial for many problems. Thus a clear dependence on $\phi$ and some uniform estimates of \eqref{eq:Ilambda} with respect to a family of phase functions are desired. 
On this regard, in one dimension a very useful tool is the Van der Corput lemma (see Stein \cite[Proposition 2 in Chapter VIII]{Stein}):
\begin{lem*}[Van der Corput]\label{lem:staph}
Assume $\phi$ is real-valued and smooth in $(a,b)$,  and
$|\phi^{(k)}(x)|\geq 1$ for all $x\in (a,b)$. Then
\[\aabs{\int_a^b e^{i\lambda \phi(x)}\psi(x)dx}\leq c_k \lambda^{-1/k}\bigg[|\psi(b)|+\int_a^b|\psi'(x)|dx\bigg]\]
holds when (1) $k\geq 2$, or (2) $k=1$ and $\phi'(x)$ is monotonic.
Here $c_k$ is a constant depending only on $k$.
\end{lem*}

This lemma is very convenient for applications. The assumption $|\phi^{(k)}(x)|\geq 1$ is a normalized condition which can be easily arranged in applications. A natural question is the analogue of the Van der Corput lemma in high dimensions. 
In \cite[Theorem 7.7.5]{Hormander}, it was proved that
\begin{align}\label{eq:Ilambda2}
\left|I(\lambda)-\frac{e^{i\lambda \phi(x_0)}\psi(x_0)(i2\pi)^{d/2}\lambda^{-d/2}}{(\det \hess \phi(x_0))^{1/2}}\right|\leq C(\norm{\phi}_{C^4})\norm{\psi}_{C^2}\lambda^{-1}.
\end{align}
The constants are explicitly tracked, but the bound $\lambda^{-1}$ is clearly not satisfactory for $d\geq 2$.  Recently, \cite{ABZ} tracked the constant in the following estimate:
\[|I(\lambda)|\leq \tilde C_{\phi,\psi}|\lambda|^{-d/2}\]
and proved 
\begin{align*}
\tilde C_{\phi,\psi}\leq C\frac{\brk{1+\sum\limits_{2\leq |\alpha|\leq d+2}\sup\limits_{x\in \supp\psi}|\p^\alpha \phi(x)|}^{d/2}}{\min\limits_{x\in \supp\psi}|\det \hess \phi(x)|}(\sum\limits_{|\alpha|\leq d+2}\sup\limits_{x\in \supp\psi}|\p^\alpha \psi(x)|).
\end{align*}
Compared to \eqref{eq:Ilambda}, the leading term does not appear in the above bound. However, for some problems we need to exploit not only the decay of $|I(\lambda)|$, but also the oscillation of $I(\lambda)$ in the leading term. Moreover, we can also exploit some cancellation such that the leading term vanishes, so that $I(\lambda)$ has better decay.  For example, the uniform decay and oscillation of the Bessel functions $J_\nu^M(r)$ as $r\to \infty$ (with respect to $\nu$) in Section \ref{subsect:Bessel} was crucial to prove the generalized Strichartz estimates by the author in \cite{GuoNon}. 

The purpose of this note is therefore to track the dependence of $C_{\phi,\psi}$ on $\phi,\psi$ in \eqref{eq:Ilambda}, which, in particular, provides more information of the oscillatory integral than the upper bound in \cite{ABZ}. 
We will prove a uniform version of \eqref{eq:Ilambda} under some normalized conditions inspired by the Van der Corput lemma.    Instead of the pointwise estimate,  we use some $L^p$-estimate for the remainder term and get various uniform estimates when the phase functions satisfy certain conditions.  This enables us to reduce the requirement of the smoothness on the phase functions,  and hence improve the results in \cite[Theorem 7.7.5]{Hormander} and also obtain a refined version of the well-known Van der Corput Lemma.   
We then present several applications of such estimates to partial differential equations and Bessel functions.

This note is organized as follows.  In Section 2, we present the one-dimensional estimate. The high-dimensional case will be presented in Section 3. Finally, in Section 4, we will apply these results to study the uniform estimates of the Bessel function and dispersive estimates in PDE.

\section{Uniform estimates for oscillatory integrals}

In this section we study $I(\lambda)$ under some normalized assumptions. These normalized conditions can be easily arranged in applications.  Let $B_r(\R^d)=\{x\in \R^d:|x|\leq r\}$. For function $f$ defined on $B_1$, $k\leq j\in \N=\{0,1,2,\cdots\}$, we define the following norms on $f$:
\[S_k(f)=\sum_{|\al|=k}\sup_{x\in B_1}|\p^\alpha f(x)|,\quad S_{[k,j]}(f)=\sum_{k\leq l\leq j}S_l(f).\]
For function $f$,  $\hat f$ or $\ft f$ denotes its Fourier transform 
\begin{align}\label{eq:Fourier}
\hat{f}(\xi)=(\F f )(\xi)=\frac{1}{(2\pi)^{d/2}}\int_{{\R}^d} e^{- ix\cdot \xi}
f(x)dx, \quad \xi\in \R^d.
\end{align}
Then $\ft^{-1} f=\F f(-\xi)$.   $A\les B$ means $A\leq CB$ for some universal constant $C$ depending only on dimension $d$. $A\sim B$ means $A\les B$ and $B\les A$.

\subsection{One-dimensional estimates}
In this subsection we first consider the one dimensional case.  For this case it is relatively easy to track the constant in \eqref{eq:Ilambda}.  We have

\begin{thm}\label{thm1}
Assume $d=1$, $\psi$ is supported in $B_1$,  and $\phi$ satisfies the following two assumptions\\
(H1): $\phi(0)=\phi'(0)=0$; \\
(H2): $\frac{1}{10}\leq \phi''(x)\leq 10$, if $|x|\leq 1$.\\
Then we have
\[I(\lambda)=(2\pi)^{1/2}\lambda^{-1/2}\frac{\psi(0)e^{i\pi/4}}{[\phi''(0)]^{1/2}}+R(\lambda)\]
and $R(\lambda)$ satisfies the following estimates:  
\begin{itemize}
\item[(1)] For $p\in [1,\infty)$,  there exists a constant $C$ independent of $\phi, \psi, \lambda$ such
\[|R(\lambda)|\leq C \lambda^{-1+\frac{1}{2p}}\brk{\norm{\phi'''}_{L^p(B_1)}\norm{\psi}_{L^\infty(B_1)}+\norm{\psi'}_{L^\infty(B_1)}}.\]

\item[(2)] $|R(\lambda)| \les  \lambda^{-3/2} C(S_{[2,4]}(\phi), S_{[0,2]}(\psi)).$

\item[(3)] If assuming further $|\psi^{(k)}(x)|\leq C_k$ and $|\phi^{(k)}(x)|\leq
C_k$ for $|x|<1$, $k\in \N$, then for any $K\in \N$ we have
\[R(\lambda)=(2\pi)^{1/2}\lambda^{-1/2}e^{i\pi/4}\sum_{k=1}^K\frac{a_k\lambda^{-k}}{k!}+\tilde R(\lambda)\]
where $|a_k|\leq C^k$ and $|\tilde R(\lambda)|\leq
C_K\lambda^{-K-3/2}$ with $C,C_K$ independent of
$\phi,\lambda$.
\end{itemize}

\end{thm}

\begin{proof}
First we use some arguments in \cite[Theorem 3.11 in Chapter 3]{Zwor}.
Since $\phi'(x)=x\cdot \int_0^1\phi''(tx)dt$,  then we have
$\phi'(x)\sim x$ if $0<|x|<1$.  We can write
$\phi(x)=\frac{1}{2}\al(x)x^2$ with
$\al(x)=2\int_0^1(1-t)\phi''(tx)dt$.  By assumption (2) we know $\al(x)\sim 1$
if $|x|<1$.  Then
\EQ{
I(\lambda)=\int e^{i\lambda \alpha(x)x^2/2}\psi(x)dx.
}
Make a change of variables
\[y:=\al(x)^{1/2}x=\pm \sqrt{2\phi(x)}, \quad \pm x>0.\]
We see that $\p_y x=\frac{x}{\phi'(x)}\al(x)^{1/2}=\frac{y}{\phi'(x)}\sim 1$ if
$|x|<1$ and $x_y(0)=\frac{1}{[\phi''(0)]^{1/2}}$. So it determines
uniquely a function $x=x(y)$.  
Using the equality
\[\p_yx=\frac{x}{\phi'(x)}\al(x)^{1/2}=(\int_0^1\phi''(tx)dt)^{-1}(2\int_0^1(1-t)\phi''(tx)dt)^{1/2},\]
we can get for $|x|\leq 1$
\EQ{\label{eq:intxyy}
|\partial_y^2x|\les \int_0^1 t|\phi'''(tx)| \partial_yx dt+\int_0^1(1-t) t|\phi'''(tx)| \partial_yx dt.
}
We have
\[I(\lambda)=\int e^{\frac{i\lambda y^2}{2}}\psi(x(y))x_y(y)dy.\]
Let $u(y)=\psi(x(y))x_y(y)$.  Then $I(\lambda)=\jb{u, e^{-\frac{i\lambda
y^2}{2}}}$. Using Parseval's formula and the fact $\ft (e^{-\frac{i\lambda
y^2}{2}})=\lambda^{-1/2}e^{-\frac{i\pi}{4}}e^{\frac{i\xi^2}{2\lambda}}$,
we get
\begin{align*}
I(\lambda)=&\lambda^{-1/2}e^{\frac{i\pi}{4}}\int
e^{-\frac{i\xi^2}{2\lambda}}\hat u(\xi)d\xi\\
=&\lambda^{-1/2}e^{\frac{i\pi}{4}}\int
\hat u(\xi)d\xi+{\lambda^{-1/2}e^{\frac{i\pi}{4}}}\int
{(e^{-\frac{i\xi^2}{2\lambda}}-1)}\hat
u(\xi)d\xi\\
=&(2\pi)^{1/2}
\lambda^{-1/2}e^{\frac{i\pi}{4}}\frac{\psi(0)}{[\phi''(0)]^{1/2}}+{\lambda^{-1/2}e^{\frac{i\pi}{4}}}\int
{(e^{-\frac{i\xi^2}{2\lambda}}-1)}\hat
u(\xi)d\xi\\
:=&(2\pi)^{1/2}
\lambda^{-1/2}e^{\frac{i\pi}{4}}\frac{\psi(0)}{[\phi''(0)]^{1/2}}+R(\lambda).
\end{align*}

Now we estimate the remainder term $R(\lambda)$ by different arguments from \cite{Zwor}.   First we show part (1) in the conclusion.  We can rewrite
\EQN{
R(\lambda)=&\lambda^{-1}e^{\frac{i\pi}{4}}\int
\frac{(e^{-\frac{i\xi^2}{2\lambda}}-1)\lambda^{1/2}}{i\xi}(i\xi)\hat
u(\xi)d\xi\\
=&\lambda^{-1}e^{\frac{i\pi}{4}}\int
\frac{(e^{-\frac{i\xi^2}{2\lambda}}-1)\lambda^{1/2}}{i\xi}\widehat
{\partial_y u}(\xi)d\xi.
}
Let $h(\xi)=\frac{e^{-i\xi^2/2}-1}{i\xi}$ and $h_\lambda(\xi)=h(\xi\lambda^{-1/2})$.   We have $h, \hat h\in L^p$ for $1<p\leq \infty$.   Indeed, 
\EQN{
\hat h(x)=&\int \frac{e^{-i\xi^2/2}-1}{i\xi}\cdot e^{-ix\xi} d\xi\\
=&\brk{\int_{|\xi|\les 1}+\int_{1\les |\xi|\les |x|}+\int_{|\xi|\gg \max(|x|,1)}}\frac{e^{-i\xi^2/2}-1}{i\xi}\cdot e^{-ix\xi} d\xi\\
:=&h_1(x)+h_2(x)+h_3(x).
}
Clearly we have $|h_1(x)|\les (1+x^2)^{-K/2}$ for any $K\in \N$ (We use a smooth cut-off for $|\xi|\les 1$).    By the Van der Corput lemma we have $|h_2(x)|+|h_3(x)|\les |x|^{-1}$ for $|x|\geq 1$.  For $|x|\leq 1$,  we have 
\EQN{
|h_3(x)|\les \aabs{\int_{|\xi|\gg 1}\frac{e^{-i(\frac{\xi^2}{2}+x\xi)}}{\xi} d\xi}+ \aabs{\int_{|\xi|\gg 1}\frac{e^{-ix\xi}}{\xi} d\xi}\les 1.
}
Thus,   we get for $p\in [1,\infty)$
\EQ{
|R(\lambda)|\les& \lambda^{-1} \norm{\widehat{h_\lambda}}_{p'}\norm{\partial_y u}_p\les \lambda^{-1} \lambda^{\frac{1}{2p}}\norm{\partial_y u}_p\\
\les& \lambda^{-1+\frac{1}{2p}}\brk{\norm{\phi'''}_{L^p(B_1)}\norm{\psi}_{L^\infty(B_1)}+\norm{\psi'}_{L^\infty(B_1)}},
}
where we used  \eqref{eq:intxyy}.

So far,  we only used the third derivative on the phase function $\phi$.   If we have estimates on higher derivative, then we have better estimates on $R(\lambda)$.
Now we prove part (2) in the conclusion.  Note that
\begin{align}\label{eq:xybound}
|\p_y^{k}(x)|\leq C(S_{[2,k+1]}(\phi)),\quad  k\geq 1.
\end{align}
Then we can also rewrite $R(\lambda)$ as
\EQN{
R(\lambda)=\frac{(2\pi
\lambda^{-1})^{1/2}e^{\frac{i\pi}{4}}}{2\lambda}\int
\frac{(e^{-\frac{i\xi^2}{2\lambda}}-1)2\lambda}{-i\xi^2}(-i\xi^2)\hat
u(\xi)d\xi.
}
Let $\tilde h(\xi)=\frac{e^{-{i\xi^2/2}}-1}{i\xi^2}$ and $\tilde h_\lambda(\xi)=\tilde h(\xi\lambda^{-1/2})$.   We have $|\ft \tilde h(x)|\les (1+|x|)^{-2}$ and hence $\ft \tilde h\in L^1$.  
By \eqref{eq:xybound} we get
\EQ{
|R(\lambda)|\les& \lambda^{-3/2} \norm{\widehat{\tilde h_\lambda}}_{1}\norm{\partial^2_y u}_\infty \les  \lambda^{-3/2} C(S_{[2,4]}(\phi), S_{[0,2]}(\psi)).
}

For part (3),  if $|\psi^{(k)}(x)|\les 1$ and $|\phi^{(k)}(x)|\les 1$ for
$|x|<1$, $k\in \N$,  then $|\p_y^{k}(x)|\les 1$ holds for $k\in \N$.
By the Taylor's expansion
$e^{-\frac{i\xi^2}{2\lambda}}=\sum_{k=0}^\infty
\frac{(-i\xi^2/(2\lambda))^k}{k!}$, we can prove the expansion for
$R(\lambda)$ with
\[a_k=2^{-k}i^k\p^{2k}u(0).\]
We complete the proof of the lemma.
\end{proof}

\begin{remark}
(1) The estimate of part (1) and (2) on $R(\lambda)$ improves the results in \cite[Theorem 7.7.5]{Hormander}.   The estimate of part (1) relies only on the third derivative and gives almost the same decay (e.g. $p=\infty$).   Assuming the bound on fourth derivative,  then we improve the decay rate $\lambda^{-1}$ to $\lambda^{-3/2}$.    This can be replaced by second derivative.  Indeed,  we have
\EQ{
\partial_y^2 x=\frac{1}{\phi'(x)}-\frac{y \phi''(x)x_y}{\phi'(x)^2}=\frac{1}{\phi'(x)}\brk{1-\frac{2\phi(x)\phi''(x)}{\phi'(x)^2}}.
}
So in part (1) of Theorem \ref{thm1},  $\norm{\phi'''}_{L^p(B_1)}$ can be replaced by $\norm{\frac{1}{\phi'(x)}\brk{1-\frac{2\phi(x)\phi''(x)}{\phi'(x)^2}}}_{L^p(B_1)}$.

(2) If $\phi$ depends on $\lambda$, $\psi$ is independent of $\lambda$,
then $x=x(y,\lambda)$, and
\[\p_\lambda x=-\frac{\p_\lambda \phi}{\p_x\phi}, \quad \p_{y\lambda }^2x=-\frac{y\p_x^2\phi\p_\lambda \phi}{(\p_x\phi)^3}.\]
Then we get
\[|\p_\lambda R(\lambda)|\les \lambda^{-5/2}+\lambda^{-3/2}\sup_{|x|\les 1}\frac{|\p_\lambda \phi|}{|x|^2}.\]

(3) We can extend Theorem \ref{thm1}  to general case.  If
$\phi$ satisfies $\phi'(x_0)=0$ for some $x_0\in \supp (\psi)$ and
$\phi'(x)\ne 0$ for $x_0\ne x\in \supp (\psi)$, then under suitable
conditions
\[I(\lambda)=(2\pi\lambda^{-1})^{1/2}|\phi''(x_0)|^{-1/2}e^{\frac{i\pi}{4}\sgn \phi''(x_0)}e^{i\lambda \phi(x_0)}\psi(x_0)+\tilde R(\lambda).\]
$\tilde R(\lambda)$ satisfies similar estimates as $R(\lambda)$.
\end{remark}

Theorem \ref{thm1} gives a refined version of Van der Corput lemma in the case $k=2$. We can extend it to general case $k\geq 3$.

\begin{thm}\label{thm1.2}
Assume $d=1$, $k\geq 2$, $\psi$ is supported in $B_1$, and $\phi$ satisfies\\
(H1) $\phi(0)=\phi'(0)=\cdots =\phi^{(k)}=0$; \\
(H2) $\frac{1}{10}\leq \phi^{(k+1)}(x)\leq 10$, if $|x|\leq 1$; \\
Then
\[I(\lambda)=\lambda^{-\frac{1}{k+1}}\frac{c_k\psi(0)Ai_k(0)}{[\phi^{(k+1)}(0)]^{\frac{1}{1+k}}}+R(\lambda)\]
where $Ai_k(\xi)=\ft^{-1}[e^{ix^{k+1}}](\xi)$,  $c_k$ is given by \eqref{eq:ck} and for $p\in (1,2]$
\EQ{|R(\lambda)|\les& \lambda^{-\frac{1}{1+k}} \lambda^{-\frac{1}{(1+k)p'}} \brk{\norm{\phi^{(k+2)}}_{L^p(B_1)}\norm{\psi}_{L^\infty(B_1)}+\norm{\psi'}_{L^\infty(B_1)}}.}
\end{thm}

\begin{proof}
Let $P_0(t)=1$.  We define $P_{k}(t)=\int_t^1 P_{k-1}(s)ds$ for $k\geq 1$. Then $P_k(t)\geq 0$ for $t\in [0,1]$ and
\EQ{
\phi(x)=&\int_0^1 P_k(t)\phi^{(k+1)}(tx)dt\cdot x^{k+1}:=\alpha_k(x)x^{k+1},\\
\phi'(x)=&\int_0^1 P_{k-1}(t)\phi^{(k+1)}(tx)dt\cdot x^{k}:=\beta_k(x)x^{k}.
}
By assumption (2) we know $\al_k(x)\sim \beta_k(x)\sim 1$
if $|x|<1$.  Then
$I(\lambda)=\int e^{i\lambda \alpha_k(x)x^{k+1}}\psi(x)dx.$
Make a change of variables
\[y:=\al(x)^{\frac{1}{k+1}}x=\pm \phi(x)^{\frac{1}{k+1}}, \quad \pm x>0.\]
We see that 
\EQ{
\p_y x=(k+1)\frac{\phi(x)^{\frac{k}{k+1}}}{\phi'(x)}\sim 1, \quad |x|<1,
}
and $x_y(0)=c_k[\phi^{(k+1)}(0)]^{-\frac{1}{1+k}}$ with
\EQ{\label{eq:ck}
c_k=(k+1)\brk{\int_0^1 P_{k-1}(t)dt}^{-\frac{1}{1+k}}.
}
So it determines
uniquely a function $x=x(y)$.  As before, we can get
\EQ{
\norm{\partial_y^2 x}_{L^p}\les \norm{\phi^{(k+2)}}_{L^p(B_1)}.
}
We have
\[I(\lambda)=\int e^{i\lambda y^{k+1}}\psi(x(y))x_y(y)dy.\]
Let $u(y)=\psi(x(y))x_y(y)$. Using Parseval's formula we have
we get
\begin{align*}
I(\lambda)=&\lambda^{-\frac{1}{k+1}}\int
Ai_k(\lambda^{-\frac{1}{k+1}} \xi)\hat u(\xi)d\xi\\
=&(2\pi)^{1/2}\lambda^{-\frac{1}{k+1}}Ai_k(0)u(0)+\lambda^{-\frac{1}{k+1}}\int
[Ai_k(\lambda^{-\frac{1}{k+1}} \xi)-Ai_k(0)]\hat
u(\xi)d\xi\\
=&(2\pi)^{1/2}\lambda^{-\frac{1}{k+1}}\frac{c_k\psi(0)Ai_k(0)}{[\phi^{(k+1)}(0)]^{\frac{1}{1+k}}}+R(\lambda)
\end{align*}
where $R(\lambda)=\lambda^{-\frac{1}{k+1}}\int
[Ai_k(\lambda^{-\frac{1}{k+1}} \xi)-Ai_k(0)]\hat
u(\xi)d\xi$.

Now we estimate $R(\lambda)$. Let $h(\xi)=\frac{Ai_k(\xi)-Ai_k(0)}{\xi}$. We first show $h\in L^p$ for $p\in (1,2]$. Since $|Ai(\xi)|\leq C$, we only need to show $|h(\xi)|\les |\xi|^{-\frac{1}{1+k}}$ when $|\xi|<1$. Indeed, for $|\xi|<1$ we have
\EQN{
h(\xi)=&(2\pi)^{-1/2}\int e^{ix^{k+1}}\frac{e^{ix\xi}-1}{\xi}dx\\
=&(2\pi)^{-1/2}\int_{|x|\leq |\xi|^{-\frac{1}{k+1}}} e^{ix^{k+1}}\frac{e^{ix\xi}-1}{\xi}dx
+(2\pi)^{-1/2}\int_{|x|\gg |\xi|^{-\frac{1}{k+1}}} e^{ix^{k+1}}\frac{e^{ix\xi}-1}{\xi}dx\\
:=&h_1(\xi)+h_2(\xi).
}
For the first term, by the Van der Corput lemma (for $k+1$-th derivatives), we have $|h_1(\xi)|\les |\xi|^{-\frac{1}{1+k}}$. For the second term, by the Van der Corput lemma (for the first derivative), we have
\EQ{
|h_2(\xi)|\les |\xi|^{\frac{k}{1+k}}|\xi|^{-1}\les |\xi|^{-\frac{1}{1+k}}.
}
Therefore, letting $h_\lambda(\xi)=h(\lambda^{-\frac{1}{1+k}}\xi)$ we have
\EQ{
|R(\lambda)|\les& \lambda^{-\frac{2}{1+k}} \int h_\lambda(\xi) \xi \hat u(\xi)d\xi\\
\les&\lambda^{-\frac{2}{1+k}} \norm{\widehat{h_\lambda}}_{p'}\norm{\partial_y u}_p\les \lambda^{-\frac{2}{1+k}} \lambda^{\frac{1}{(1+k)p}}\norm{\partial_y u}_p\\
\les& \lambda^{-\frac{1}{1+k}} \lambda^{-\frac{1}{(1+k)p'}} \brk{\int |\psi(x(y))x_{yy}(y)|^pdy+\int |\psi'(x(y))x^2_{y}(y)|^pdy}^{1/p}\\
\les& \lambda^{-\frac{1}{1+k}} \lambda^{-\frac{1}{(1+k)p'}}  \brk{\norm{\phi^{(k+2)}}_{L^p(B_1)}\norm{\psi}_{L^\infty(B_1)}+\norm{\psi'}_{L^\infty(B_1)}}.
}
The proof is complete. 
\end{proof}

Next we consider the high-dimensional cases.  In contrast to the one-dimensional case, 
it is difficult to place the normalized conditions. We can only normalize one condition.  An analogue of the Van der Corput lemma in high dimensions may be complicated. We only consider the non-degenerate case, namely at the critical point, the hessian is non-degenerate.

\subsection{High-dimensional estimates}

For a symmetric matrix $\mA=(a_{ij})$, we use two matrix norms $\norm{\mA}=\Lr(\R^d,\R^d)$ and $|\mA|_2=(\sum\limits_{1\leq i,j\leq d}|a_{ij}|^2)^{1/2}$. We denote by $\Lambda(\mA)$ to be the set of the $d$ eigenvalues of $\mA$. Then
\begin{align*}
\norm{\mA}=\sup_{\lambda\in \Lambda(\mA)} |\lambda|, \qquad |\mA|_2=(\sum_{\lambda\in \Lambda(\mA)} |\lambda|^2)^{1/2}.
\end{align*}
We have $\norm{\mA}\leq |\mA|_2\leq d^{1/2}\norm{\mA}$ and if $\det(\mA)\neq 0$
\begin{align}
\norm{\mA^{-1}}^{-d}\leq |\det (\mA)|\leq \norm{A}^d,\quad \norm{\mA^{-1}}\leq \frac{\norm{\mA}^{d-1}}{|\det \mA|}.
\end{align}
Next we need the implicit function theorem.  The theorem is classical and standard.  We include here a proof to track the dependence on the function.

\begin{lem}[Implicit function]\label{lem:implicit}
Assume $f(x,y):B_1(\R^n)\times B_1(\R^m)\to \R^m$ is a smooth function. Assume
\[f(0,0)=0, \quad \mA=\p_yf(0,0) \mbox{ is non-degenerate}\]
and $\sup_{x,y}(|\p_x f|+|\p_x\p_y f|+|\p_y^2f|)\leq K$. 
Then $\exists r\geq \frac14 (m+n)^{-6}|\mA|_2^{2-2m}|\det \mA|^2 K^{-2}$ and a unique continuous function $\varphi:B_r(\R^n)\to \R^m$ such that $\varphi(0)=0$ and $f(x,\varphi(x))=0$ for $x\in B_r(\R^n)$. Moreover, $\sup_{x\in B_r}|\varphi(x)|\leq \sqrt{r}$.
\end{lem}
\begin{proof}
Define the mapping
\[T\varphi(x)=\varphi(x)-\mA^{-1}f(x,\varphi(x)).\]
We will show $T:X\to X$ is a contraction mapping on the metric space
\[X=\{\varphi\in C(B_r(\R^n):\R^m): \varphi(0)=0, \norm{\varphi}\leq \delta\}\]
with the metric $d(\phi,\psi)=\norm{\phi-\psi}=\sup_{x\in B_r}|\phi(x)-\psi(x)|$. Here the parameters $r,\delta$ is to be determined.

Clearly, for $\varphi\in X$, $T\varphi(0)=0$ and since $f(x,\varphi)=f(x,\varphi)-f(0,\varphi)+f(0,\varphi)-f(0,0)=\p_xf(\xi_x,\varphi)x+\mA \varphi+\frac12\p_y^2f(0,\xi_y)\varphi^2$, then we get  
\begin{align*}
\norm{T\varphi}\leq& \norm{\mA^{-1} \p_xf(\xi_x,\varphi)x}+\frac{1}{2}\norm{\mA^{-1}\p_y^2f(0,\xi_y)\varphi^2}\\
\leq&\norm{\mA^{-1}}\cdot mnKr+\frac12\norm{\mA^{-1}}\cdot m^3K\delta^2
\end{align*}
and for $\phi,\psi\in X$, $f(x,\varphi)-f(x,\psi)=\p_x\p_y f(\xi_x,\varphi_y)(\varphi-\psi)x-[f(0,\varphi)-f(0,\psi)]=\p_x\p_y f(\xi_x,\varphi_y)(\varphi-\psi)x+\mA (\varphi-\psi)+\frac12\p_y^2f(0,\xi_y)(\varphi-\psi)^2$ from which we get
\begin{align*}
\norm{T\varphi-T\psi}\leq&\norm{\mA^{-1}}\cdot m^2nKr\norm{\varphi-\psi}+\frac12\norm{\mA^{-1}}\cdot m^3K\delta \norm{\varphi-\psi}.
\end{align*}
Taking 
$\delta=\frac12(\norm{\mA^{-1}}(m+n)^3K)^{-1}$ and 
$r=\delta^2$, we get $T:X\to X$ is a contraction mapping and hence prove the lemma. 
\end{proof}

\begin{cor}
Assume $f$ is defined on $B_1(\R^d)$, smooth, real-valued, $f(0)=0$ and $\mA=\hess f(0)$ is non-degenerate. Then at least one of the following holds:

\noindent (1) $\exists x_0\in B_1$ such that $\nabla f(x_0)=0$;

\noindent (2) $\exists \delta=|\det \mA|^2\delta(S_2(f))$ such that $|\nabla f(x)|\geq |\det \mA|^2C(S_2(f))$, $\forall x\in B_{\delta}$. 

\noindent Here $\delta(\cdot)$ and $C(\cdot)$ are both continuous and increasing functions from $\R^+$ to $\R^+$ depending only on $d$. 
\end{cor}
\begin{proof}
Let $y_0=\nabla f(0)$. Define the function $F(x,y)=\nabla f(y)+x-y_0$. Then $F(0,0)=0$ and $\p_yF(0,0)=\hess f(0)$. By Lemma \ref{lem:implicit} we get there exists $r\geq |\det \mA|^2 C(S_2(f))$ and a continuous function $y=\varphi(x)$ satisfies $F(x,\varphi(x))=0$ for $x\in B_r(\R^d)$. In particular, if $|\nabla f(0)|\leq |\det \mA|^2C(S_2(f))$, then $y_0\in B_r$ and hence (1) holds by taking $x_0=\varphi(y_0)$.

On the other hand, if $|\nabla f(0)|\geq |\det \mA|^2C(S_2(f))$, then by the mean value formula
\[\nabla f(x)=\nabla f(0)+\hess f(x_\theta)x\]
we get that if $|x|\leq |\det \mA|^2\delta(S_2(f))$
\[|\nabla f(x)\geq \frac12 |\nabla f(0)|, \quad x\in B_\delta.\]
Thus (2) holds.
\end{proof}

\begin{lem}[Morse's Lemma]\label{lem:morse}
Assume $d\geq 2$, $f$ is defined on $B_1$, smooth, real-valued, and satisfies

\noindent (1) $f(0)=0$, $\nabla f(0)=0$;

\noindent (2) $\mA=\hess f(0)$  is non-degenerate. 

\noindent Then there exists $\delta\geq C_d{S_3(f)^{-1}|\mA|_2^{1-d}}|\det \mA|$ and a diffeomorphism $\gamma: B_{\de}\to \gamma(B_{\de})\subset B_1$ such that $
\ga(0)=0$ and
\[f(\gamma (x))=x^T\mA x, \quad x\in B_{\delta}.\]
Moreover, $B_{\de/1000}\subset \gamma(B_{\de})$ and for any $\al$ 
\EQ{\label{eq:dgamma}
\sup_{x\in B_{\de}}|\p_x^\al \gamma(x)|\leq C_{\al,d}(S_{[2, 2+|\al|]}(f)).
}
\end{lem}
\begin{proof}
We define the quadratic function $\wt A(x)=x^T\mA x$ for $x=(x_1,\cdots,x_d)^T\in \R^d$. We can write $f(x)-\wt A(x)=x^T\mG_x x$
where
\begin{align*}
\mG_x =& \int_0^1(1-s)\hess f(sx)ds-\mA\\
= &-\frac{1}{2}\mA+\int_0^1(1-s)[\hess f(sx)-\hess f(0)]ds\\
:= &-\frac{1}{2}\mA+\wt \mG_x.
\end{align*}
Next we define a path from $\wt A$ to $f$:
\[f_t(x)=\wt A(x)+t[f(x)-\wt A(x)], \quad t\in [0,1].\]
Since $\nabla f_t(0)=0$, then 
\[\nabla f_t(x)=\int_0^1 \hess f_t(sx)ds \cdot x=\brk{\mA+t\int_0^1 [\hess f(sx)-\hess f(0)]ds}\cdot x:=\mB_{t,x} \cdot x.\]
We write $\mB_{t,x}=\mA+\wt\mB_{t,x}$ with $\wt\mB_{t,x}=t\int_0^1 [\hess f(sx)-\hess f(0)]ds$.
Note that $\wt\mB_{t,0}=\wt\mG_0=0$ and $\max(\norm{\wt\mB_{t,x}},\norm{\wt\mG_{x}})\leq dS_3(f)\delta$ for all $t\in [0,1]$ and $x\in B_\delta(\R^d)$. 

We write $\mB_{t,x}=\mA(\mI+\mA^{-1}\wt\mB_{t,x})$.  Since $\norm{\mA^{-1}\wt\mB_{t,x}}\leq \norm{\mA^{-1}}dS_3(f)\delta\leq 1/100$,  $\mB_{t,x}$ is invertible and 
\begin{align}
\mB_{t,x}^{-1}=(\mI+\mA^{-1}\wt\mB_{t,x})^{-1}\mA^{-1}.
\end{align}
Indeed, $(\mI+\mA^{-1}\wt\mB_{t,x})^{-1}=\sum_{k=0}^\infty (-\mA^{-1}\wt\mB_{t,x})^k$ and we have
\[\norm{(\mI+\mA^{-1}\wt\mB_{t,x})^{-1}}\leq 10/9.\]
Define a vector field $\xi(t,x)=-\mB_{t,x}^{-1}\mG_x x=\frac{1}{2}x+\wt \xi(t,x)$. Then we get 
\begin{align}\label{eq:xibd}
\sup_{|x|\leq \de, t\in [0,1]}|\xi(t,x)|\leq& \frac{1}{2}\cdot \frac{1+2d\delta \norm{\mA^{-1}}}{1-d\delta \norm{\mA^{-1}}}\leq \frac{3}{4},\\
\sup_{|x|\leq \de, t\in [0,1]}|\p_x\wt\xi(t,x)|\leq& \norm{\mB_{t,x}^{-1}\mG_x+\frac{1}{2}\mI}+\norm{\mB_{t,x}^{-1}\p\mG_x x}+\norm{\mB_{t,x}^{-2}\p \mB_{t,x}\mG_x x}\\
\leq& C_d\norm{\mA^{-1}}S_{3}(f)\delta\leq \frac{1}{30}
\end{align}
Moreover, for $|\al|\geq 2$, since
$\max(|\p_x^\al \mG_x|_2,|\p_x^\al \mB_{t,x}|_2)\leq d \cdot S_{2+|\al|}(f)$, we get
\begin{align*}
|\p_x^\al \xi(t,x)|\leq& C_{\al,d}\cdot \sum_{1\leq |\be|\leq |\al|}\norm{\mA^{-1}}^{1+|\beta|}\cdot [S_{3+|\al|-|\be|}(f)]^{|\be|}\\
\leq &C_{\al,d}(S_{[2, 2+|\al|]}(f))
\end{align*}
For fixed $x$ let $\varphi_t(x)$ be the flow generated by $\xi$, namely
\[\frac{d}{dt}\varphi_t=\frac{1}{2}\varphi_t+\wt\xi(t,\varphi_t), \quad \varphi_0=x.\]
By \eqref{eq:xibd} and the Picard-Lindel\"of theorem, we get for $|x|\leq \frac{1}{4}$
the flow $\varphi_t$ exists for $t\in [0,1]$ and $x\to \varphi_t(x)$ is a differmorphism for $t\in [0,1]$. We have $f_t(\varphi_t(x))=f_0(x)=\wt A(x)$ since
\[\frac{d}{dt}[f_t(\varphi_t)]=\dot{f}_t(\varphi_t)+\nabla f_t(\varphi_t)\cdot \dot\varphi_t=\varphi_t^T\mG_{\varphi_t}\varphi_t+\xi(t,\varphi_t)^TB_{t,\varphi_t}\varphi_t=0.\]
In particular, $y=\varphi_1(x): B_{1/4}\to \varphi_1(B_{1/4})\subset B_1$ is a diffeomorphism and satisfies $f(\varphi_1(x))=\wt A(x)$. By uniqueness $\varphi_t(0)=0$ for $t\in [0,1]$. Moreover, from
\begin{align}\label{eq:varphi_t}
\nabla \varphi_t(x)=e^{1/2}\mI+\int_0^t e^{(t-s)/2} \nabla \wt\xi(s,\varphi_s)\nabla \varphi_s(x)ds
\end{align}
we get
$\norm{(\nabla \varphi_t(x))^{-1}}\leq 10$
and $B_{\de/1000}\subset \varphi_1(B_{\de})$. By \eqref{eq:varphi_t} again and Gronwall inequality, we get
\[|\p_x^\al\varphi_t(x)|\leq C(\sup_{|x|\leq \de,t\in [0,1]}|\p_x^\al \xi(t,x)|)\leq C_{\al,d}(S_{[2, 2+|\al|]}(f)), \quad t\in [0,1].\]
Taking $\gamma=\varphi_1$ we finish the proof.
\end{proof}

\begin{thm}\label{thm2}
Assume $d\geq 2$, $\psi$ is supported in $B_{1/1000}$, and $\phi$ is defined on $B_1$ and satisfies\\
(1) $\phi(0)=0, \nabla\phi(0)=0$; \\
(2) $\hess \phi(0)$ is non-degenerate and $S_2(\phi)\leq C_2$, $|\det \hess \phi(0)|S_3(\phi)^{-1}\geq C_3$; \\
(3) $|S_k(\phi)|\leq C_4$ for some universal $C_4$, if $|x|\leq 1$, $0\leq k\leq [d/2]+4$.\\
Then
\[I(\lambda)=(-\lambda \pi^{-1})^{-d/2}e^{-i\frac{\pi}{4} \sgn \mA}|\det \mA|^{-1/2}\psi(0)+R(\lambda),\]
and for some $C$ depending on $C_i$, but  independent of $\phi, \psi, \lambda$, we have
\[|R(\lambda)|\leq C
\lambda^{-d/2} \lambda^{-1/2}\lambda^{\frac{d}{4}-\frac{[d/2]}{2}}  |\det \mA|^{-1/2}\norm{\mA^{-1}}S_{[0,[d/2]+3]}(\psi).\]
\end{thm}

\begin{proof}
Let $\mA=\hess \phi(0)$. By Lemma \ref{lem:morse},  there exists a differmorphism $\gamma: B_{1/4}\to \gamma(B_{1/4})$ such that
\[\phi(\gamma(y))=y^T\mA y.\]
Making the change of variables $x=\gamma(y)$, we get
\[I(\lambda)=\int e^{i\lambda y^T\mA y}\psi(\ga(y))|\det \p_y \ga(y)|dy.\]
As the one dimension case, using the fact (see (3.2.3) in \cite{Zwor})
\[\F(e^{-i\lambda y^T\mA y})=(-\lambda \pi^{-1})^{-d/2}e^{i\frac{\pi}{4} \sgn \mA}|\det \mA|^{-1/2}e^{\frac{i}{4\lambda}\xi^T \mA^{-1}\xi},\]
and denoting $u(y)=\psi(\ga(y))|\det \p_y \ga(y)|$, by Plancherel's equality we get
\begin{align*}
I(\lambda)=&\int \wh u(\xi)(-\lambda \pi^{-1})^{-d/2}e^{-i\frac{\pi}{4} \sgn \mA}|\det \mA|^{-1/2}e^{-\frac{i}{4\lambda}\xi^T \mA^{-1}\xi}d\xi\\
=&(-\lambda \pi^{-1})^{-d/2}e^{-i\frac{\pi}{4} \sgn \mA}|\det \mA|^{-1/2}u(0)\\
&+\int (-\lambda \pi^{-1})^{-d/2}e^{-i\frac{\pi}{4} \sgn \mA}|\det \mA|^{-1/2}
(e^{-\frac{i}{4\lambda}\xi^T \mA^{-1}\xi}-1)\hat u(\xi)d\xi\\
:=&(-\lambda \pi^{-1})^{-d/2}e^{-i\frac{\pi}{4} \sgn \mA}|\det \mA|^{-1/2}\psi(0)+R(\lambda).
\end{align*}
Now we estimate $R(\lambda)$.   Take $K=[d/2]+1$.  We have
\EQ{
R(\lambda)=(-\lambda \pi^{-1})^{-d/2}e^{-i\frac{\pi}{4} \sgn \mA}|\det \mA|^{-1/2}\int 
\frac{(e^{-\frac{i}{4\lambda}\xi^T \mA^{-1}\xi}-1)}{|\xi|^K}|\xi|^K\hat u(\xi)d\xi
}
By \eqref{eq:dgamma} we get
\begin{align*}
|R(\lambda)|\les& \lambda^{-d/2}|\det \mA|^{-1/2}\normo{\frac{e^{-\frac{i}{4\lambda}\xi^T \mA^{-1}\xi}-1}{|\xi|^K}}_{L^2}\norm{\p^K u}_{L^2}\\
\les&
 \lambda^{-d/2}|\det \mA|^{-1/2}\lambda^{-K/2}\normo{\frac{e^{-\frac{i}{4\lambda}\xi^T \mA^{-1}\xi}-1}{|\lambda^{-1/2}\xi|^K}}_{L^2}\norm{\p^K u}_{L^2}\\
 \les&
\lambda^{-d/2} \lambda^{-K/2+d/4}|\det \mA|^{-1/2}\norm{\mA^{-1}} \norm{\p^K u}_{L^2}\\
\les&
\lambda^{-d/4-K/2} |\det \mA|^{-1/2}\norm{\mA^{-1}}C(S_{2,K+3}(\phi))S_{[0,K+2]}(\psi).
\end{align*}
We complete the proof of the lemma.
\end{proof}
%
%
%

\section{Applications}

\subsection{Uniform estimates on Bessel function}\label{subsect:Bessel}

In this subsection, we give an application of our results to the Bessel function
\begin{align*}
J_\nu(r)=\frac{(r/2)^\nu}{\Gamma(\nu+1/2)\pi^{1/2}}
\int_{-1}^1e^{irt}(1-t^2)^{\nu-1/2}dt, \ \ \nu>-1/2,
\end{align*}
where $\Gamma(\cdot)$ is the Gamma function.
Bessel functions is one of the most important special functions.  It has been extensively studied,  for example, see \cite{Bess, NIST}.   We focus on the decay and oscillation properties.  For a fixed argument $\nu$,  the Bessel function are better understood.   For example,  we have (see \cite{Stein}): as $r\to \infty$
\begin{align}
J_\nu (r)|\leq C_\nu |r|^{-1/2}
\end{align}
and 
\begin{align}\label{eq:Besselex}
J_\nu (r)=(\frac{\pi r}{2})^{-1/2}\cos(r-\frac{\pi}{4}-\frac{\nu \pi}{2})+O_\nu(1) r^{-3/2}.
\end{align}
However,  for the uniform properties (with respect to all $\nu$),  the problems are much harder and there are less results.   For this purpose we usually use the Schl\"{a}fli's integral representation of Bessel function (see p. 176, \cite{Bess}):
\begin{align}\label{eq:Besselint}
J_\nu(r)=&\frac{1}{2\pi}\int_{-\pi}^\pi e^{i(r\sin x-\nu
x)}dx-\frac{\sin(\nu\pi)}{\pi}\int_0^\infty
e^{-\nu\tau-r\sinh \tau}d\tau \nonumber\\
:=&J_\nu^M(r)-J_\nu^E(r),
\end{align}
We have the uniform decay (see section 10.20.4 in \cite{NIST}): assume $r,\nu>10$,  then
\begin{align}
|J_\nu(r)|+|J_\nu'(r)|\leq& Cr^{-1/3}(1+r^{-1/3}|r-\nu|)^{-1/4}.
\end{align}
and (see Lemma 2.2 in \cite{GuoNon}): assume $\nu\in \N$, $\nu>r+\lambda$, and
$\lambda>r^{\frac{1}{3}+\e}$ for some $\e>0$,  then for any $K\in \N$
\begin{align}
|J_\nu(r)|+|J_\nu'(r)|\leq C_{K,\e}r^{-K\e}.
\end{align}

The main contribution of $J_\nu(r)$ is on the region $r>\nu+\nu^{1/3}$.   We would like to obtain a uniform estimates that provide both decay and oscillation as \eqref{eq:Besselex},  especially in the
transitive region $\nu+\nu^{1/3}<r<2\nu$.   These properties play crucial roles in the study of Schr\"odinger equations. 
We have

\begin{lem}[Asymptotical property]\label{lem:Bessel}\footnote{This lemma was given in \cite{GuoNon} without proof.    We include the proof here as an application.  The proof is taken from the arxiv version of  \cite{GuoNon} and not published anywhere.}
Let $\nu>10$ and $r>\nu+\nu^{1/3}$. Then

(1) We have
\[J_\nu(r)=\frac{1}{\sqrt{2\pi}}\frac{e^{i\theta(r)}+e^{-i\theta(r)}}{(r^2-\nu^2)^{1/4}}+h(\nu,r),\]
where
\[\theta(r)=(r^2-\nu^2)^{1/2}-\nu \arccos \frac{\nu}{r}-\frac \pi 4\]
and
\[|h(\nu,r)|\les \bigg(\frac{\nu^2}{(r^2-\nu^2)^{7/4}}+\frac{1}{r}\bigg)1_{[\nu+\nu^{1/3},2\nu]}(r)+r^{-1}1_{[2\nu,\infty)}(r).\]

(2) Let $x_0=\arccos \frac{\nu}{r}$. For any $K\in \N$ we have
\begin{align*}
h(\nu,r)=&(2\pi)^{-1/2}e^{i\theta(r)}x_0\sum_{k=1}^K\frac{
(rx_0^3)^{-k-1/2}a_k(x_0)}{k!}\\
&+(2\pi)^{-1/2}e^{-i\theta(r)}x_0\sum_{k=1}^K\frac{(rx_0^3)^{-k-1/2}\tilde
a_k(x_0)}{k!}+\tilde h(\nu,r)
\end{align*}
with functions $|\p^l a_k|+|\p^l \tilde a_k|\les 1$ for any $l\in
\N$ and
\[|\tilde h(\nu,r)|\les \bigg(\frac{r^{\frac{K}{2}+\frac{1}{4}}}{(r-\nu)^{\frac{3K}{2}+7/4}}+\frac{1}{r}\bigg)1_{[\nu+\nu^{1/3},2\nu]}(r)+r^{-1}1_{[2\nu,\infty)}(r).\]
Moreover, if $\nu\in \Z$, we have better estimate
\[|\tilde h(\nu,r)|\les \frac{r^{\frac{K}{2}+\frac{1}{4}}}{(r-\nu)^{\frac{3K}{2}+7/4}}1_{[\nu+\nu^{1/3},2\nu]}(r)+r^{-3/2}1_{[2\nu,\infty)}(r).\]

\end{lem}
\begin{proof}
Part (1) was given in \cite{BC} without a proof.   Here we give a proof by Theorem
\ref{thm1}. If $\nu\in \Z$, then $J_\nu^E(r)=0$. Thus it
suffices to consider $J_\nu^M(r)$. Denote $\phi(x)=\sin
x-\frac{\nu}{r}x$. Let $\phi'(x)=\cos x-\frac{\nu}{r}=0$, then we
find two solutions $x=\pm x_0=\pm \arccos \frac{\nu}{r}$. Since
$\nu<r$, we get $x_0\sim \frac{\sqrt{r^2-\nu^2}}{r}<1$. We divide
the proof into two cases.

{\bf Case 1.} $r\geq 2\nu$.

In this case we have $x_0\sim 1$. Let $\beta(x)$ be a cutoff
function around $0$ and supported in $\{|x|\ll 1\}$. Let $\tilde
\beta=1-\beta(x-x_0)-\beta(x+x_0)$. Then
\[J_\nu^M(r)=\frac{1}{2\pi}\int_{-\pi}^\pi e^{ir\phi(x)}[\beta(x-x_0)+\beta(x+x_0)+\tilde \beta(x)]dx:=I_1+I_2+I_3.\]
First, we estimate the term $I_3$. Since $|\phi'(x)|\sim 1$ in
$\supp \tilde\beta$, integrating by part we get that
\begin{align*}
|I_3|\les |\int_{-\pi}^\pi \frac{\p_x[e^{ir\phi(x)}]}{ir\phi'(x)}
\tilde\beta(x)dx|\les r^{-1}.
\end{align*}
If $\nu\in \Z$, we can do better since the boundary term vanishes.
Indeed, in this case from the fact that
$e^{ir\phi(\pi)}=e^{ir\phi(-\pi)}$, $\phi'(\pi)=\phi'(-\pi)$,
$\tilde\beta(\pi)=\tilde\beta(-\pi)$, we can get
\begin{align*}
|I_3|\les r^{-2}, \quad \mbox{ if } \nu\in \Z.
\end{align*}
Now we consider the term $I_1$. We have
\[I_1=\frac{1}{2\pi}\int e^{ir\phi(x+x_0)}\beta(x)dx.\]
It is easy to check that $\phi(x+x_0)-\phi(x_0),\beta$ satisfy the
conditions in Theorem \ref{thm1}. Thus by Theorem \ref{thm1} we get
\[I_1=\frac{1}{\sqrt{2\pi}}\frac{e^{i\theta(r)}}{(r^2-\nu^2)^{1/4}}+R_1(\nu,r)\]
with $|R_1|\les r^{-3/2}$. Similarly, for $I_2$ we have
\[I_2=\frac{1}{\sqrt{2\pi}}\frac{-e^{i\theta(r)}}{(r^2-\nu^2)^{1/4}}+R_2(\nu,r)\]
with $|R_2|\les r^{-3/2}$. Therefore, we prove part (1) by setting
$h=R_1+R_2+I_3+J_\nu^E$.

{\bf Case 2.} $r<2\nu$.

Let $\gamma=1-\beta(\frac{x-x_0}{x_0})-\beta(\frac{x+x_0}{x_0})$.
Then
\[J_\nu^M(r)=\frac{1}{2\pi}\int_{-\pi}^\pi e^{ir\phi(x)}[\beta(\frac{x-x_0}{x_0})+\beta(\frac{x+x_0}{x_0})+\gamma(x)]dx:=II_1+II_2+II_3.\]
First, we estimate the term $I_1$. We have
\begin{align*}
II_1=\frac{1}{2\pi}\int
e^{ir\phi(x)}\beta(\frac{x-x_0}{x_0})dx=\frac{x_0}{2\pi}\int
e^{irx_0^3\cdot x_0^{-3}\phi(x_0x+x_0)}\beta(x)dx.
\end{align*}
By the condition $r>\nu+\nu^{1/3}$ we get $rx_0^3\ges 1$. Let
$\tilde\phi(x)=x_0^{-3}[\phi(x_0x+x_0)-\phi(x_0)]$. By the mean
value formula we can verify the conditions in Theorem \ref{thm1} for
$\tilde\phi(x),\beta$. Thus by Theorem \ref{thm1} we get
\[II_1=\frac{1}{\sqrt{2\pi}}\frac{e^{i\theta(r)}}{(r^2-\nu^2)^{1/4}}+\tilde R_1\]
with $|\tilde R_1|\les x_0(rx_0^3)^{-3/2}\les
\frac{\nu^2}{(r^2-\nu^2)^{7/4}}$. Similarly, for $II_2$ we get
\[II_2=\frac{1}{\sqrt{2\pi}}\frac{-e^{i\theta(r)}}{(r^2-\nu^2)^{1/4}}+\tilde R_2\]
with $|\tilde R_2|\les x_0(rx_0^3)^{-3/2}\les
\frac{\nu^2}{(r^2-\nu^2)^{7/4}}$.

Now we estimate the term $II_3$. We have
\[II_3=\frac{1}{2\pi}\int
e^{ir\phi(x)}\eta(x)\gamma(x)dx+\frac{1}{2\pi}\int_{-\pi}^{\pi}
e^{ir\phi(x)}(1-\eta(x))dx:=II_3^1+II_3^2.\] For the term $II_3^1$,
it's easy to see that $|\phi'(x)|\ges x_0^2,\,
|\p_x^k(\frac{1}{\phi'(x)})|\les x_0^{-2-k}$, $\forall k\in \N$ for
$x\in \supp (\eta\gamma)$, then integrating by parts we get that
\[|II_3^1|\les r^{-K}\bigg|\int
e^{ir\phi(x)}(\p_x
\frac{1}{\phi'(x)})^K\big[\eta(x)\gamma(x)\big]dx\bigg|\les x_0
(rx_0^3)^{-K}\les \frac{\nu^2}{(r^2-\nu^2)^{7/4}}.\] For the term
$II_3^2$, we have $|\phi'(x)|\sim 1$ for $x\in \supp(1-\eta)$. Thus
we get $|II_3^2|\les r^{-1}$ using integration by parts. If $\nu\in
\Z$, as in case 1, the boundary value vanishes, and we get
$|II_3^2|\les r^{-2}$. Thus we prove part (1) by setting $h=\tilde
R_1+\tilde R_2+II_3+J_\nu^E$.

Now we prove part (2). We only need to consider $\tilde R_1,\tilde
R_2$ in case 2. By Theorem \ref{thm1} we have
\begin{align*}
\tilde
R_1=&x_0(2\pi)^{-1/2}e^{ir\phi(x_0)}(rx_0^3)^{-1/2}e^{i\pi/4}\sum_{k=1}^K\frac{a_k(rx_0^3)^{-k}}{k!}+x_0O((rx_0^3)^{-K-3/2})\\
=&(2\pi)^{-1/2}e^{i\theta(r)}x_0\sum_{k=1}^K\frac{a_k
(rx_0^3)^{-k-1/2}}{k!}+x_0O((rx_0^3)^{-K-3/2}).
\end{align*}
We can obtain the expansion for $\tilde R_2$ similarly. We complete
the proof.
\end{proof}

\subsection{Dispersive estimates}
We consider the dispersive equation
\EQ{
i\partial_t u+\omega(-i\nabla) u=0, \quad u(0)=\phi
}
where $\omega: \R^d\to \R$ and $\omega(-i\nabla) =\ft^{-1}\omega(\xi)\ft$.   Using the Fourier transform,  we have \[u=S(t)\phi:=C\int_{\R^d}e^{ix\xi+it\omega(\xi)}\hat f(\xi)d\xi.\] 
Let $\chi$ be a smooth cutoff function adapted to $[1,2]$ and define $\dot P_k=\ft^{-1}\chi(\frac{|\xi|}{2^k})\ft$.
The dispersive estimate of the following type
\EQ{\label{eq:dispersive}
\norm{S(t)\dot P_k f}_{L^\infty}\leq C(k)|t|^{-\theta}\norm{f}_{L^1},
}
has been playing a fundamental role in the study of the (nonlinear) dispersive equation, e.g.  Strichartz estimates,  well-posedness and asymptotic behaviour.    When $\omega$ is radial,  the dispersive estimate \eqref{eq:dispersive} was systematically studied in \cite{decay} under some assumptions of $\omega$ and its derivatives at $0$ and $\infty$.

In this paper we try to continue \cite{decay} by allowing more general dispersion $\omega$.   By the Young inequality,  \eqref{eq:dispersive} can be derived from the following oscillatory integral estimate
\EQ{
\left|\int_{\R^d}e^{it\big(\omega(2^k\xi)+\frac{2^kx\xi}{t}\big)} \chi(\xi)d\xi\right|\leq C(k)2^{-kd} |t|^{-\theta},
}
which involves a phase function depending on $x,k$ and $t$.   
Let $$\Phi(\xi)=\Phi_{x,k,t}(\xi)=2^{k\theta}\brk{\omega(2^k\xi)+\frac{2^kx\xi}{t}}$$ for $\theta\in \R$.   Note that
\EQ{
\nabla \Phi(\xi)=2^{k\theta}\brk{2^k \nabla\omega(2^k\xi)+\frac{2^kx}{t}},\quad
\hess \Phi(\xi)=2^{k\theta}2^{2k}\brk{\omega_{ij}(2^k\xi)}.
}
If we can chose $\theta$ such that $\Phi$ satisfies the conditions in Theorem \ref{thm2},  then we can obtain the following estimate 
\EQ{
\left|\int_{\R^d}e^{it\big(\omega(2^k\xi)+\frac{2^kx\xi}{t}\big)} \chi(\xi)d\xi\right|\le 2^{-k\theta d/2}|t|^{-d/2}.
}
In particular, 
when $\omega(\xi)=h(|\xi|)$ is a radial function,  then
\EQ{
\omega_{ij}(\xi)=\brk{h''(|\xi|)-\frac{h'(|\xi|)}{|\xi|}}\frac{\xi_i\xi_j}{|\xi|^2}+\frac{h'(|\xi|)}{|\xi|}\delta_{ij}.
}
We have
\[\det (w_{ij})=h''(|\xi|), \quad \det \hess \Phi(\xi)=2^{k\theta d}2^{2kd}h''(2^k|\xi|).\]
When $h''$ is non-zero,  this was handled in \cite{decay}.   Here we give some examples that $h''$ may have zeros. 

\begin{ex} $\omega(\xi)=h(|\xi|)=\sqrt{|\xi|+|\xi|^3}$,  $\xi\in \R^d$.   This appears in the water-wave system,  see \cite{DIPP}.

Note that 
\EQ{
h'(r)=&\frac{1+3r^2}{2\sqrt{r+r^3}}\\
h''(r)=&\frac{3(r^2+1)^2-4}{4(r(1+r^2))^{3/2}}.
}
Then $h''(r)=0$ has a unique solution at $r=r_0=\sqrt{\frac{2}{\sqrt{3}}-1}$.   So the main contribution of the dispersive estimate is from the frequency of size $r_0$.   Therefore
\EQN{
\int_{\R^d}e^{it\brk{\omega(\xi)+\frac{x\xi}{t}}}\chi(\xi)d\xi
=&\int_{||\xi|-r_0|\ges \delta} e^{it\brk{\omega(\xi)+\frac{x\xi}{t}}}\chi(\xi)d\xi+\int_{||\xi|-r_0|\les \delta} e^{it\brk{\omega(\xi)+\frac{x\xi}{t}}}\chi(\xi)d\xi\\
:=& I+II.
}
We may assume $|x|\ges |t|$.
For the term $I$,  since $|h''(r)|\ges \delta$ for $|r-r_0|\ges \delta$,  then using polar coordinates and Theorem \ref{thm1} we get
\EQ{
|I|\les& \aabs{\int_{||\xi|-r_0|\ges \delta} e^{it\brk{\omega(\xi)+\frac{x\xi}{t}}}\chi(\xi)d\xi}\\
\les& \aabs{\int_{|\rho-r_0|\ges \delta} e^{ith(\rho)}\chi(\rho) \rho^{d-1} (\rho |x|)^{-\frac{d-2}{2}}J_{\frac{d-2}{2}}(\rho |x|)d\rho}\\
\les& (|t|\delta)^{-1/2}|t|^{-\frac{d-1}{2}}\les |t|^{-d/2}\delta^{-1/2}.
}
For the term $II$,  using polar coordinates and decay of the Bessel functions,  we get
\EQ{
|II| \les& \aabs{\int_{|\rho-r_0|< \delta} e^{ith(\rho)}\chi(\rho) \rho^{d-1} (\rho |x|)^{-\frac{d-2}{2}}J_{\frac{d-2}{2}}(\rho |x|)d\rho}\\
\les&\delta |t|^{-\frac{d-1}{2}}.
}
Optimising $\delta$ by taking $\delta=t^{-1/3}$,  we obtain
\EQ{
\aabs{\int_{\R^d}e^{it\brk{\omega(\xi)+\frac{x\xi}{t}}}\chi(\xi)d\xi}\les  |t|^{-\frac{d-1}{2}-\frac{1}{3}}.
}
\end{ex}

\begin{ex}
$\omega(\xi)=h_2(|\xi|)=|\xi|\sqrt{\frac{2+|\xi|^2}{1+|\xi|^2}}$.  This appears in the Euler-Poisson system (see \cite{GP}).   Similarly, $h_2''(r)=$ has a unique solution at $r=r_0=\sqrt{1+\sqrt{7}}$.
The rest follows in the same way as Example 1 and we omit the details. 
\end{ex}

\section*{Acknowledgement}
The author thanks Xiaolong Han and Melissa Tacy for the discussions and comments that improve the writing of this paper.


\begin{thebibliography}{99}

\bibitem{ABZ}
T. Alazard,  N. Burq and C. Zuily,  ‘A stationary phase type estimate’,  Proc.  Amer. Math.  Soc.  145 (2017),  2871--2880.

\bibitem{BC}{J. Barcelo, A. Cordoba, Band-limited functions: $L^p$-convergence, Trans. Amer. Math. Soc. 312 (1989), 1--15.}

\bibitem{DIPP}
Y.  Deng,  A.  D.  Ionescu,  B.  Pausader and F.  Pusateri,  
Global solutions of the gravity-capillary water-wave system in three dimensions, Acta Math., 219 (2017),  213--402. 

\bibitem{GP}
Y. Guo and B.  Pausader,  Global Smooth Ion Dynamics in the Euler-Poisson System,  Commun. Math. Phys.  303 (2011),  89--125.  

\bibitem{GuoNon}	Z. Guo,  Sharp spherically averaged Strichartz estimates for the Schrodinger equation,  Nonlinearity 29 (2016), 1668--1686.

\bibitem{GHT}
Z. Guo,  X. Han and M. Tacy,  $L^p$ bilinear quasimode estimates,  J.  Geom.  Anal.  29,  No. 3 (2019),  2242--2289.

\bibitem{decay}
Z. Guo, L. Peng and B. Wang, Decay estimates for a class of wave equations, Journal of Functional Analysis,  254/6 (2008), 1642-1660.

\bibitem{Hormander} L. H\"ormander, The Analysis of Linear Partial Differential Operators I, Springer, 2015.

\bibitem{KPV}{C. E. Kenig,  G. Ponce and L. Vega,  Oscillatory integrals and regularity of dispersive equations,  Indiana University Mathematics Journal
Vol.  40,  No. 1 (1991),  33-69. }

\bibitem{NIST} F. W. J. Olver, D. W. Lozier, R. F. Boisvert, and C. W. Clark,  NIST Handbook of Mathematical Functions,  Cambridge University Press,  New York,  NY, 2010. 

\bibitem{Stein}
{\sc E. M. Stein}, { Harmonic analysis: real-variable methods,  orthogonality, and oscillatory integrals},  Princeton Mathematical Series 43,  Princeton University Press,  Princeton,  NJ,  1993.

\bibitem{Bess}{G. Watson, A treatise on the theory of Bessel functions, Reprint of the second (1944) edition.
Cambridge University Press, Cambridge, 1995.}

\bibitem{Zwor}
M.  Zworski, Semiclassical Analysis,  American Mathematical Society, Providence (2012).


\end{thebibliography}
\end{document}